\documentclass[11pt]{amsart}
\usepackage{amssymb,mathrsfs,enumitem,amsmath,bbold}
\usepackage[mathscr]{euscript}
\usepackage{color}
\usepackage[all]{xy}
\definecolor{Chocolat}{rgb}{0.36, 0.2, 0.09}
\definecolor{BleuTresFonce}{rgb}{0.215, 0.215, 0.36}
\definecolor{EgyptianBlue}{rgb}{0.06, 0.2, 0.65}

\usepackage[colorlinks,final,hyperindex]{hyperref}
\hypersetup{citecolor=BleuTresFonce, urlcolor=EgyptianBlue, linkcolor=Chocolat}

\usepackage{fourier}

\newtheorem{theorem}{Theorem}[section]

\newtheorem{proposition}{Proposition}[section]

\theoremstyle{definition}

\newtheorem*{definition}{Definition}

\DeclareFontFamily{U}{wncy}{}
\DeclareFontShape{U}{wncy}{m}{n}{<->wncyr10}{}
\DeclareSymbolFont{mcy}{U}{wncy}{m}{n}
\DeclareMathSymbol{\Sha}{\mathord}{mcy}{"58}

\DeclareMathAlphabet{\pazocal}{OMS}{zplm}{m}{n}

\DeclareMathAlphabet{\mathbbold}{U}{bbold}{m}{n}

\def\k{\mathbbold{k}}

\DeclareMathOperator{\Der}{Der}
\DeclareMathOperator{\Hom}{Hom}
\DeclareMathOperator{\supp}{supp}

\DeclareMathOperator{\DGA}{\mathsf{DGA}}

\begin{document}

\title{Tangent complexes and the Diamond Lemma}
\author{Vladimir Dotsenko} 
\address{Institut de Recherche Math\'ematique Avanc\'ee, UMR 7501, Universit\'e de Strasbourg et CNRS, 7 rue Ren\'e-Descartes, 67000 Strasbourg, France}
\email{vdotsenko@unistra.fr}
\author{Pedro Tamaroff} 
\address{School of Mathematics, Trinity College, Dublin 2, Ireland}
\email{pedro@maths.tcd.ie}

\begin{abstract}
The celebrated Diamond Lemma of Bergman gives an effectively verifiable criterion of uniqueness of normal forms for term rewriting in associative algebras. We present a new way to interpret and prove this result from the viewpoint of homotopical algebra. Our main result states that every multiplicative free resolution of an algebra with monomial relations gives rise to its own Diamond Lemma, so that Bergman's 
condition of ``resolvable ambiguities'' becomes the first non-trivial component of the Maurer--Cartan equation in the corresponding tangent complex. The same approach works for many other algebraic structures, emphasizing the relevance of computing multiplicative free resolutions of algebras with monomial relations.    
\end{abstract}


\maketitle

\section*{Introduction}

\subsection*{Context of our work}
When studying algebras presented by generators and relations, the central general result is the statement known as the Diamond Lemma. Historically, one would say that it was already implicit in Shirshov's work \cite{MR0183753} who proved an analogous but technically more involved result, the Composition Lemma, in the case of Lie algebras. About a decade after that, this result was proved in the case of associative algebras by Bokut in \cite{MR0506423}, and independently by Bergman~\cite{Bergman}, who also proposed the name ``Diamond Lemma'' in order to emphasize the analogy with the classical result of Newman \cite{MR7372}. Paraphrasing the opening phrase of \cite{Bergman}, the main results of our paper are doubly trivial: Bergman's Diamond Lemma is, in his own words, trivial, and our main goal is to explain a new trivial proof of this trivial statement, as well as some other trivial statements, from the point of view of homotopical algebra. However, we believe that readers of this paper can benefit from it in a number of ways. For a reader whose intuition comes from homotopical algebra, our proof will hopefully feel like a conceptual explanation of useful but seemingly technical criteria of ``resolvable ambiguities'' for uniqueness of normal forms. For a reader with background in Gr\"obner bases or term rewriting, our proof will offer intuition behind both the Diamond Lemma and its optimisation, known as the Chain Criterion in the commutative case \cite{MR575678,10.1145/2631948.2631968,10.1145/1088261.1088267} and as the Triangle Lemma in the case of noncommutative associative algebras~\cite{MR3642294,Lat88,MR1360005}, as well as precise guidance on how to generalise those for other algebraic structures. Specifically, our work means that computing models of algebras with monomial relations explicitly helps both to state the relevant Diamond Lemmas and to optimise them. 

\subsection*{From Gr\"obner bases to deformation theory and back}
The idea of using Gr\"obner bases for computing homological invariants of associative algebras is well known. The seemingly earliest instance appears in the work of Priddy \cite{MR265437} who constructed resolutions for the ground field, viewed as the trivial module, for algebras presented by quadratic--linear relations; this was later generalised by Anick~\cite{MR846601} to arbitrary presentations.  Existence of such resolutions is not at all surprising for the following reason. Defining relations of an associative algebra $A$ form a Gr\"obner basis if the same monomials give a vector space basis for both algebras $A$ and $A_{\mathrm{mon}}$; the latter is the algebra with monomial relations given by the leading monomials of the relations of $A$ with respect to a suitable ordering. The trivial module for $A_{\mathrm{mon}}$ admits a combinatorially defined free resolution, implicit in the work \cite{MR551760} of Backelin, and explicitly determined by Green, Happel and Zacharia in \cite{MR769766}. Such resolution can be obtained as a contraction of the bar resolution; incorporating the lower terms of relations from a Gr\"obner basis can obtained from such a contraction by homological perturbation techniques \cite{MR3276839,MR1007895,MR1103672,MR1187288}. In fact, the one-sided module resolution of Anick can be generalised to a bimodule resolution, as established by Bardzell \cite{MR1874282,MR2694031}, leading to a computational method for determining the Hochschild (co)homology of an algebra, and thus for studying deformations of a given algebra. This construction works under a weaker assumption of a convergent rewriting system instead of a Gr\"obner basis, as shown by Chouhy and Solotar in \cite{MR3334140}.   

Our main wish, motivated by interest in generalising these methods to algebraic structures other than associative algebras, is to go in the opposite direction and re-discover effective criteria for reduction systems with unique normal forms from the deformation theory viewpoint. Since in those situations the multiplication table of an algebra $A$ is obtained from the simple combinatorial multiplication table of $A_{\mathrm{mon}}$ by incorporating appropriate lower terms, we naturally find ourselves in the framework of deformation theory. In fact, very recently the bimodule resolutions mentioned above have been used to study deformations of algebras with monomial relations by Barmeier and Wang in their work on deformation theory of quiver algebras~\cite{BarWang} and by Redondo and Rossi Bertone in~\cite{RedBer}, using the following idea. A free bimodule resolution \[A_{\mathrm{mon}}\otimes C_\bullet\otimes A_{\mathrm{mon}}\simeq A_{\mathrm{mon}}\] leads to the representative  $\Hom(C_\bullet,A_{\mathrm{mon}})$ of the deformation complex, and one can use a strong deformation retract relating it to Hochschild complex to compute explicitly its $L_\infty$-algebra structure using the homotopy transfer theorem. Studying deformations amounts to studying solutions to the Maurer--Cartan equation in that $L_\infty$-algebra, and one may obtain various results this way. However, if one adopts this viewpoint, the Diamond Lemma criterion of ``resolvable ambiguities'' cannot be recovered instantly: it involves a calculation in the \textsl{free} associative algebra which cannot be reproduced directly since the target space of the deformation complex is the monomial algebra $A_{\mathrm{mon}}$. The two deformation complexes are homotopy equivalent, so the necessary result can be proved in principle. However, recovering the classical Diamond Lemma will require some translation, in the spirit of Anick's slogan ``the perturbative construction of the second differential of the resolution is precisely the procedure of resolving ambiguities'', see \cite[Sec.~2]{MR846601}. We also refer the reader to a discussion of this phenomenon in the survey of Ufnarovski \cite[Sec.~3.8]{MR1360005} where one also finds the Triangle Lemma as a way to optimise the algorithm.

\subsection*{Deformations via tangent complexes}
It turns out that a natural way to remedy the situation is to move from homology to homotopy, and study deformations of algebras in terms of homotopical algebra. This means working with multiplicative free resolutions, that is resolutions which are free as algebras rather than as bimodules. In this case, the homotopy class of the deformation complex of an algebra can be represented by its tangent complex, that is, the differential graded Lie algebra of derivations of its free resolution. Historically, this approach to deformations first emerged in the case of commutative associative algebras. In this context, it makes sense to note that some of the pioneering works both in the theory of Gr\"obner bases and in the deformation theory appeared in the algebro-geometric context: the former in the solution of the resolution of singularities problem in characteristic zero by Hironaka \cite{MR0175898,MR0199184}, and the latter in Palamodov's work on deformation of analytic spaces \cite{MR0508121}, where the tangent complex approach is attributed to unpublished work of Tyurina (who tragically died in a kayaking accident at the age of 32). However, the standard bases of Hironaka were just one of many tools in a paper of more than two hundred pages, and the power of this method in algebraic geometry and commutative algebra became apparent only after work of Buchberger \cite{MR2202562} who highlighted their algorithmic nature. The tangent complex approach to deformation theory of algebras became widely known from the famous manuscript \cite{SchSta} by Schlessinger and Stasheff. We would like to also remark that a seemingly completely independent path to multiplicative resolutions which however stays away from deformation theory questions emerges in  higher-categorical rewriting theory. Original work of Squier \cite{MR1146597,MR920522} on homological finiteness conditions for monoids admitting a convergent presentation first received a higher-categorical flavour in works of Lafont \cite{MR1324032} and Citterio \cite{MR1897811}; corresponding multiplicative resolutions appear in work of Lafont and M\'etayer~\cite{MR2498787} and Guetta \cite{guetta2020homology} relying on theory of resolutions of categories by polygraphs developed by M\'etayer \cite{MR1988395}. This formalism was recently extended to the $\k$-linear context by Guiraud, Hoffbeck, and Malbos~\cite{MR4002273}, who in particular explained how to construct a polygraphic resolution of an associative algebra by techniques similar to the perturbative construction of $A_\infty$-structures~\cite{MR885535} and free resolutions of algebras~\cite{MR1109665}. 

\subsection*{Structure of the paper}
For us, it is crucial that working with free resolutions means that recovering Bergman's criterion of resolvable ambiguities stands a chance, at least in principle, since our deformation complexes of algebras with monomial relations have free associative algebras as target spaces. Our approach blends ideas and methods from several different areas, and we tried to offer enough detail when recalling the necessary background information. Our first homotopical re-formulation of uniqueness of normal forms is a general criterion in terms of perturbations of free resolutions (Theorem \ref{th:Perturb}). Using that result, we prove another general criterion in terms of approximate solutions to the Maurer--Cartan equation in the tangent complex (Theorem \ref{th:MCLift}). Using this latter result, we establish the main result of this paper (Theorem \ref{thm:Diamond}): every model extending the Shafarevich complex of the given relations gives rise to its own Diamond Lemma, so that the resolvable ambiguities condition of the Diamond Lemma becomes the first non-trivial component of the Maurer--Cartan equation in the tangent complex of the model. Applying that latter result to some specific combinatorial models, we recover the classical Diamond Lemma in the case of the inclusion--exclusion model of the first author and Khoroshkin \cite{MR3084563} and the Triangle Lemma in the case of the minimal model of the second author~\cite{Tam1}. We discuss some analogues and possible generalisations of our results for other types of algebras towards the end of the article. Already in the case of operads with monomial relations, explicit formulas for minimal models are not known in general, and only a much weaker version of the Triangle Lemma is available~\cite{MR3642294}; we hope that our work will attract due attention to this question.  

\subsection*{Precursors: Evgeny Solomonovich Golod and Victor Nikolaevich Latyshev}
To conclude the introduction, we would like to mention two crucial sources of inspiration for our work. The first of them is the work of Golod who discovered a proof of the Diamond Lemma \cite{Golod} using the non-commutative analogue of the Koszul complex, the Shafarevich complex associated to a system of elements in the free algebra originally introduced in \cite{MR0161852}. That complex is a differential graded algebra whose homotopy type depends on the presentation of the original algebra. Our homotopy invariant free resolutions are obtained from the Shafarevich complex at the cost of adding extra generators of higher degrees; we believe that the benefit of the resulting clarity outweighs the cost. Our second inspiration comes from the work of Latyshev who used normal forms to resolve some particular cases \cite{MR0142595,MR0156874} of the celebrated Specht problem on identities of associative algebras \cite{MR35274}, and advocated general importance of normal forms and standard bases \cite{MR1754671,MR2128915,MR2744977}. Both Golod and Latyshev passed away relatively recently (in July 2018 and April 2020, respectively). We would like to dedicate this work to their memory.

\subsection*{Acknowledgements. } The authors thank Leonid Arkadievich Bokut, Eric Hoffbeck, Anton Khoroshkin and Muriel Livernet for useful discussions. The first author also thanks Dmitri Piontkovski who introduced him to Shafarevich complexes many years ago, and Emil Sk\"oldberg for sharing some computations  of minimal models of monomial operads in low homological degrees. The idea of this article emerged during the workshop ``Homotopy meets homology'' at Hamilton Mathematics Institute (Trinity College Dublin) in May 2019, and was primarily guided by our wish to unveil homotopical algebra behind the work~\cite{RedBer} presented at that workshop by Mar\'ia Julia Redondo; we are grateful to  the Science Foundation Ireland and the Simons Foundation for allocating funding that made that workshop possible. Part of this work was done during the second author's visits to Strasbourg funded by IRMA (Institut de Recherche Math\'ematique Avanc\'ee), he is thankful for both the financial support of those visits and the excellent working conditions. 

\section{Conventions}

Unless otherwise indicated, all objects in this paper are defined over an arbitrary ground field $\k$. By a \emph{graded} vector space we mean a vector space $V$ of the form 
 \[
V=\bigoplus_{n\in\mathbb{Z}} V_n ,
 \]
where we write $|v|=n$ for $v\in V_n$, and refer to $n$ as the \emph{homological degree} of an element~$v$. The adjective `homological' means that these degrees create ``Koszul signs'' arising from exchanging factors in tensor products. Recall that one defines the tensor product of two graded vector spaces by the formula 
 \[
(V\otimes W)_n:=\bigoplus_{i+j=k}V_i\otimes W_j 
 \]
summarised in words by ``degrees add under tensor products'', and homological degrees enter the formula for the symmetry isomorphism 
 \[
\tau_{V,W}\colon V\otimes W\to W\otimes V
 \]
given by $\tau(v\otimes w)=(-1)^{|v||w|} w\otimes v$. (This, for a trained eye, creates signs in a lot of places; for example, if one applies the tensor product of linear maps to a tensor product of two vectors, the formula $(\phi\otimes\psi)(v\otimes w)=(-1)^{|\psi||v|} \phi(v)\otimes \psi(w)$ has to be used.) Occasionally, our graded vector spaces will have extra gradings which do not create any extra signs in formulas; in such cases, we shall use the word ``grade'' for such degrees.

\section{Term rewriting}\label{sec:Rewriting}

In this section, we give a short introduction to term rewriting in the linear context. Mathematically, all of this goes back to~\cite{Bergman}; terminologically, we choose to follow the recent literature on rewriting systems, see, for example, \cite{MR4002273} and references therein. 

\subsection{Rewriting systems}
Let us fix once and for all a finite set $X$; we shall denote by $\langle X\rangle$ the free monoid generated by $X$, so that the linear span $\k\langle X\rangle$ is the free associative $\k$-algebra generated by $X$. For an element $g\in \k\langle X\rangle$ we shall denote by $\supp(g)$ the set of all elements of $\langle X\rangle$ that appear in $g$ with a non-zero coefficient. 

\begin{definition}[Rewriting system]
A \emph{rewriting system} on $\k\langle X\rangle$ is a triple $(X,W,f)$, where $W\subset\langle X\rangle$, and $f$ is a function on $W$ with values in $\k\langle X\rangle$. 
\end{definition}

The right way to think of a rewriting system is as of a collection of rules that allow to rewrite each monomial $w\in W$ into the element $f(w)\in \k\langle X\rangle$. In this way, each rewriting system $\Sigma$ naturally gives rise to an associative algebra $A_\Sigma$ with relations $R=\{w-f(w) \colon w\in W\}$: 
 \[
A_\Sigma=\k\langle X\mid R \rangle .
 \]
Each rewriting rule of $\Sigma$ becomes a way to replace elements of $\k\langle X\rangle$ by other representatives of the same coset in $A_\Sigma$, as follows. 

\begin{definition}[Reduction]
Suppose that $w\in W$, and $a,b\in\langle X\rangle$. A \emph{basic reduction} associated to the triple $(a,w,b)$ is the $\k$-linear endomorphism $\rho$ of $\k\langle X\rangle$ such that for each $m\in\langle X\rangle$ we have
 \[
\rho(m)=
\begin{cases}
af(w)b, \ \ \text{if } m=awb,\\
\quad\ \  m \qquad \text{otherwise. }
\end{cases}
 \]
A finite sequence of basic reductions is called a \emph{reduction}; it defines a $\k$-linear endomorphism of $\k\langle X\rangle$ by composing, in the given order, the basic reductions appearing in it.
\end{definition}

Using reductions is intended to ``simplify'' representatives of cosets. This is formalised by the notion of a normal form.

\begin{definition}[Normal form]
An element $g\in \k\langle X\rangle$ is called \emph{irreducible} if $\rho(g)=g$ for every basic reduction $\rho$. A \emph{normal form} of an element $g\in\k\langle X\rangle$ is an irreducible element $\bar{g}$ such that the cosets of $g$ and $\bar{g}$ in $A_\Sigma$ are equal.
\end{definition}

In principle, it is possible that one can perform basic reductions infinitely many times. We shall only work with rewriting systems for which it does not happen. To formalise that, we give the following definition.

\begin{definition}[Terminating rewriting system]
A \emph{pseudo-reduction} is an infinite sequence 
 \[
\rho =(\rho_1, \rho_2, \rho_3,\ldots) 
 \]
of basic reductions. To every pseudo-reduction $\rho$ and every $g \in\k\langle X\rangle$ we can 
associate the sequence 
 \[ 
\rho(g) = (\rho_1(g),\rho_2\rho_1(g),\rho_3\rho_2\rho_1(g),\ldots) 
 \]
of elements in $\k\langle X\rangle$. We say that an element $g \in\k\langle X\rangle$ is \emph{reduction finite} if for every pseudo-reduction
$\rho$, the associated sequence $\rho(g)$ is eventually constant. A rewriting system $\Sigma$ is \emph{terminating} if all elements $g\in \k\langle X\rangle$ are reduction finite.  
\end{definition}

Note that if an element $g$ is reduction finite, there exists a reduction that sends this element to an irreducible one, for otherwise we could easily find a pseudo-reduction $\rho$ for which $\rho(g)$ is not eventually constant. Thus, for a terminating rewriting system, each element has at least one normal form. All rewriting systems considered in this paper are assumed to be terminating. We shall now recall how to recast termination in terms of partial orderings.

\begin{definition}[Compatible rewriting ordering]
For a given rewriting system $\Sigma$, we say that a partial order $\leqslant$ on $\langle X\rangle$ is a \emph{rewriting ordering compatible with $\Sigma$} if the following conditions hold:
\begin{enumerate}
\item \emph{multiplicativity}: if $m<m'$ for two elements $m,m'$, then $amb<am'b$ for any $a,b\in\langle X\rangle$;
\item \emph{well-ordering}: every non-empty set of monomials has a minimal element;
\item \emph{compatibility with $\Sigma$}: for all $w\in W$ and all $w'\in\supp(f(w))$, we have $w'<w$.
\end{enumerate}
\end{definition}

Among rewriting orderings compatible with $\Sigma$ (if they exist at all), there exists the weakest possible one obtained as follows. We first define $<_\Sigma$ as the smallest transitive binary relation that is multiplicative and compatible with~$\Sigma$. If $<_\Sigma$ is a well-ordering, the reflexive closure $\leqslant_\Sigma$ is a rewriting ordering compatible with~$\Sigma$, and each other rewriting ordering compatible with~$\Sigma$ obviously refines the ordering $\leqslant_\Sigma$. According to \cite{Bergman,MR3334140}, the rewriting system $\Sigma$ is terminating if and only if $<_\Sigma$ is a well-ordering.

For a terminating rewriting system, the associative algebra
 \[
A_{\Sigma_0}=\k\langle X\mid W \rangle 
 \]
obtained from the rewriting system $\Sigma_0=(W,0)$ and called the \emph{monomial algebra associated to $\Sigma$} is a relevant object of study. Indeed, an element $g\in \k\langle X\rangle$ is irreducible if and only if no monomial $m\in\supp(g)$ is divisible by elements of~$W$, so irreducible monomials form a basis of the algebra $A_{\Sigma_0}$. However, since the same element may have several different normal forms, the cosets of irreducible monomials might not be linearly independent in the algebra $A_\Sigma$.

\begin{definition}[Convergent rewriting system]
A (terminating) rewriting system is \emph{convergent} if each element $g \in\k\langle X\rangle$ has exactly one normal form. 
\end{definition}

Among the rewriting systems, the best known class is given by \emph{Gr\"obner bases}: those are rewriting systems where there is a total order $\leqslant$ on $\langle X\rangle$ satisfying the above three conditions. Let us recall an example \cite{MR3642294,MR4002273} showing that sometimes using a partial order is really advantageous. For that, we consider the algebra $A=\k\langle x,y,z\mid x^3+y^3+z^3-xyz\rangle$. Note that for a total multiplicative order, we always have $xyz<\max(x^3,y^3,z^3)$, and one can check that for such an order the corresponding rewriting system is not convergent, so one needs to extend it by further rewriting rules. However, one can consider the rewriting system $\Sigma=(X,W,f)$ with $X=\{x,y,z\}$, $W=\{xyz\}$ and $f(xyz)=x^3+y^3+z^3$ which does not arise from a total multiplicative order. Let us explain why this rewriting system is terminating. For that, we define, for each monomial $m$, $\Phi(m)=3A(m)+B(m)$, where $A(m)$ is the number of divisors of $m$ equal to $xyz$ and $B(m)$ is the number of divisors of $m$ equal to $y$. Then one can check that any application of our rewriting rule replaces a monomial by a linear combination of terms for which the parameter $\Phi$ is strictly smaller, which guarantees the termination. Later we shall see that this rewriting system is convergent, showing that in this case there is an advantage over the theory of Gr\"obner bases.

Bergman's Diamond Lemma \cite{Bergman} establishes a criterion of convergence in terms of the so called ``resolvable ambiguities''. If this were a paper on rewriting systems, at this point we would formulate and prove that result. However, our goal is to explain how such criterion can be found, so that one can determine its analogues for other algebraic structures. It is easy to see that a rewriting system $\Sigma$ is convergent if and only if the corresponding irreducible monomials are linearly independent. In other words, $\Sigma$ is convergent if and only if the linear map $A_{{\Sigma_0}}\to A_\Sigma$ sending each element of the standard monomial basis of $A_{{\Sigma_0}}$ to its coset in $A_\Sigma$ is an isomorphism of vector spaces. This suggests that convergent rewriting systems may be studied from the viewpoint of deformation theory of associative algebras. In general, the notion of a \emph{deformation} implies that we \emph{deform} the structure by adding lower terms with respect to a suitable filtration. For example, for the classical deformation theory set-up \cite{MR2483955} defined over an Artinian local ring containing $\k$, the filtration may be defined by powers of the maximal ideal $\mathfrak{m}$ of that ring. In our case, the filtration is defined out of the partial ordering $<_\Sigma$, or, better to say, out of any order that extends it to a total well-ordering. The multiplication table in the algebra~$A_{{\Sigma_0}}$ is very simple: the product of two basis elements is either a basis element or zero. For a convergent rewriting system, the product of two basis elements in the algebra $A_{{\Sigma}}$ is either a basis element or the linear combination of smaller basis elements obtained by term rewriting. Thus, saying that $\Sigma$ is convergent is equivalent to saying that the algebra $A_\Sigma$ is a deformation of the algebra~$A_{\Sigma_0}$ with respect to the corresponding filtration. This suggests that using deformation theory of algebras might shed light on convergence of rewriting systems. The main part of this paper uses this viewpoint extensively.

\section{Homotopy theory of associative algebras}

Deformation theory of algebras borrows a lot of intuition from algebraic topology: one may say that algebraic tools for studying continuous deformations of spaces are very much amenable to the case of deformations of algebras. In algebraic topology, an important invariant of a space is its cohomology algebra. However, that algebra itself does not capture many important homotopy invariants. It turns out that to remedy that one should either consider cohomology together with certain higher structures, or work with a bigger algebra of the same homotopy type, for example the differential graded algebra of singular cochains. This led Quillen to the general philosophy of homotopical algebra~\cite{MR0223432} suggesting that one should extend categories of algebras to better behaved ``model categories'' with a notion of weak equivalence, an abstraction of homotopy equivalence, and consider the homotopy category formed by equivalence classes. Alternatively, one may say that instead of studying an algebra it is often beneficial to study its model, which is a better behaved algebra of the same homotopy type in the model category. This section gives a recollection of the corresponding definitions to make the rest of the paper more readable to experts in Gr\"obner bases and rewriting.  

\subsection{Differential graded associative algebras} A sufficiently general recipe to define model categories is to consider simplicial algebras of the given type. However, in the case of associative algebras, it turns out to be possible to do homotopy theory in a slightly more hands-on way, using differential graded algebras.

\begin{definition}[Differential graded associative algebra]
A \emph{differential graded associative algebra}, is a pair $(A,d)$ where $A$ is a graded associative algebra, that is, a graded vector space equipped with the associative degree zero map $A\otimes A\to A$, and $d\colon A\to A$ is a degree $-1$ map satisfying the condition $d^2=0$ as well as the graded derivation property $d(a_1a_2)=d(a_1)a_2+(-1)^{|a_1|} a_1 d(a_2)$. 
\end{definition}

We denote by $\DGA$ the category whose objects are differential graded associative algebras and whose morphisms from $(A,d_A)$ to $(B,d_B)$ are (degree zero) algebra homomorphisms $f\colon A\to B$ that are chain maps, that is satisfy $f\circ d_A=d_B\circ f$. Such a morphism $f$ induces a morphism $f_\bullet\colon H_\bullet(A,d_A)\to H_\bullet(B,d_B)$ of graded associative algebras obtained by computing homology. We say that $f$ is a \emph{quasi-isomorphism} if the induced homology morphism $f_\bullet$ is an isomorphism. 

\subsection{Models of algebras} As indicated above, our general plan is to replace an associative algebra with a better behaved differential graded algebra of the same homotopy type. A mathematically precise meaning of the words ``better behaved'' is given by the notion of a cofibrant object in a closed model category. 
A motivated reader is invited to consult \cite{MR1650134,MR0223432} for foundations of the general theory of model categories. According to Hinich \cite[Th.~4.1.1]{MR1465117}, the category $\DGA$ admits a closed model category structure for which weak equivalences are quasi-isomorphisms, and the fibrations are surjections. In fact, Hinich works in the general framework of algebras over operads, and in the only case relevant for this paper, that of differential graded algebras concentrated in non-negative homological degrees, the same result is established by Jardine \cite{MR1478701}. To describe cofibrant objects, we need one more definition.

\begin{definition}[Triangulated quasi-free algebra]
A \emph{quasi-free algebra} is a differential graded algebra $(B,d_B)$ for which the underlying graded associative algebra is free, that is $B\cong T(V)$ for some graded vector space $V$. Such algebra is said to be \emph{triangulated} if its space of generators admits an extra decomposition into a direct sum (of graded vector spaces) 
 \[
V\cong \bigoplus_{j\ge 1} V^{(j)} ,
 \]
for which the differential of an element $v\in V^{(j+1)}$ belongs to the subalgebra generated by $\bigoplus_{1\le k\le j} V^{(k)}$ for all $j\ge 0$. 
\end{definition}

The cofibrant objects for Hinich's model category structure on $\DGA$ are exactly the retracts of triangulated quasi-free algebras. We remark that it is useful to think that triangulated algebras are an abstraction of the construction of ``killing cycles'' of Tate \cite{MR268172,MR86072}. The relevance of triangulated differential graded \emph{commutative} algebras in homotopy theory became apparent after the seminal paper of Sullivan \cite{MR646078}; for this reason, the are frequently called Sullivan algebras.  

\begin{definition}[Model of an algebra]
A \emph{model} of a differential graded algebra $(A,d_A)$ is a triangulated quasi-free algebra $(T(V),d)$ equipped with a surjective quasi-isomorphism $f$ to $(A,d_A)$. Such a model is said to be \emph{minimal} if its differential is \emph{decomposable}, that is for each $v\in V$, the element $d(v)$ is a combination of products of length at least two in the tensor algebra. 
\end{definition}

While minimal models of algebras are unique up to isomorphism, their existence is a much more subtle matter; see, for example, the unpublished note of Keller~\cite{Kel}. However, the target algebras $(A,d_A)$ of interest for us are in fact non-graded and non-differential, in other words, $A_n=0$ for $n\ne 0$ and $d_A=0$. Since we work over a field, minimal models for such algebras exist under very mild assumptions (see, for example, \cite{MR431172,MR1109665}) which hold in all cases that are of interest to us.

\subsection{Shafarevich complexes}

We shall now recall the notion of a Shafarevich complex of a system of elements in an algebra. Those complexes are very rarely models themselves, but they offer a good starting point for constructing a model, which will be one of the guiding principles for our main results. 

\begin{definition}[Shafarevich complex]
Let $A$ be an associative (non-differential non-graded) algebra. The \emph{Shafarevich complex} of a subset $S\subset A$ is the differential graded algebra
 \[
\Sha(S,A)=(A * T(U), d) . 
 \]
Here $A * T(U)$ is the free product of $A$ and the tensor algebra, $U$ is the graded vector space whose only non-zero component is $U_1\cong \k S$, and $d$ is the unique graded derivation satisfying $d(a)=0$ for $a\in A$ and $d(e_s)=s\in A$ for each basis element $s\in U_1$.
\end{definition}

It is an immediate consequence of the definition that the degree zero homology of $\Sha(S,A)$ is isomorphic to the quotient $\cong A/(S)$. In this paper, we shall only consider the Shafarevich complexes $\Sha(S,\k\langle X\rangle)$ for the free algebra $A=\k\langle X\rangle$; it is precisely those that were first defined in work of Golod and Shafarevich \cite{MR0161852} and later called the Shafarevich complexes by Golod in \cite{Golod}. In this particular case, the degree zero homology is the algebra~$\k\langle X\mid S\rangle$. Of course, $\Sha(S,\k\langle X\rangle)=(T(V),d)$ where $V$ is the graded vector space whose only non-zero components are $V_0=\k X$ and $V_1\cong \k S$, and $d$ is the unique graded derivation satisfying $d(x)=0$ for $x\in X$ and $d(e_s)=s\in \k\langle X\rangle\cong T(V_0)$ for each basis element $s\in V_1$. Thus, $\Sha(S,\k\langle X\rangle)$ is a triangulated quasi-free algebra whose space of generators is concentrated in homological degrees zero and one corresponding to killing certain cycles in $T(V_0)$. In that capacity, it was rediscovered by Lemaire~\cite{MR500930}, and used by Anick in his work on Hilbert series of associative algebras~\cite{MR644015,MR677714}. 

In general, the positive degree homology of $\Sha(S,\k\langle X\rangle)$ does not vanish, so that differential graded algebra is not a model of $\k\langle X\mid S\rangle$. In fact, it is known that the homology of $\Sha(S,\k\langle X\rangle)$ is generated by elements of homological degrees zero and one \cite{MR1799542}, but this remarkable result will not be immediately useful for us. We however will be mainly interested in models that extend Shafarevich complexes.

\begin{definition}[Acyclic extension of the Shafarevich complex]
Let $A\cong\k\langle X\mid S\rangle$ be an associative algebra presented by generators and relations. We say that a model of the algebra $A$ is an \emph{acyclic extension of the Shafarevich complex~$\Sha(S,\k\langle X\rangle)$} if it is concentrated in non-negative homological degrees, and in homological degrees zero and one is isomorphic to the Shafarevich complex~$\Sha(S,\k\langle X\rangle)$. 
\end{definition}

\section{Rewriting systems and perturbation of models}

Since any model of an associative algebra is quasi-free, it is not particularly interesting as an algebra: all the complexity is hidden in the differential. Thus, if we replace an algebra by a model, it is reasonable to expect that studying deformation theory is done via studying perturbations of the differential. The main result of this section confirms this expectation: essentially, it says that once we fix a filtration, deforming a monomial algebra with respect to that filtration is equivalent to perturbing the differential of its model. 

\begin{definition}[Word-homogeneous algebra]
Let us call a differential associative graded algebra $(B,d)$ \emph{word-homogeneous} if it has an extra grading by the free monoid $\langle X\rangle$, that is we have a decomposition into a direct sum of chain subcomplexes
 \[
(B,d)=\bigoplus_{u\in \langle X\rangle} (B_u,d) ,
 \]
such that $B_{u_1}B_{u_2}\subseteq B_{u_1u_2}$ for all $u_1,u_2\in \langle X\rangle$. For an element $b\in B_u$, we say the element $b$ is word-homogeneous of~\emph{grade}~$u$.
\end{definition}

As above, let us consider a rewriting system $\Sigma=(X,W,f)$. Clearly, the non-differential non-graded algebra $(A_{\Sigma_0},0)$ is word-homogeneous. It is easy to establish that its minimal model is word-homogeneous; while we shall consider models that are not necessarily non-minimal, we shall restrict ourselves to word-homogeneous models when considering monomial algebras. Moreover, since the algebra~$A_{\Sigma_0}$ is concentrated in homological degree zero, it is enough to work with word-homogeneous models concentrated in non-negative homological degrees. For such a model $(T(V),d)$, the differential graded subalgebra $(T(V_0),0)$ surjects onto $A_{\Sigma_0}$; in other words, the image of the vector space $V_0$ generates the algebra $A_{\Sigma_0}$. In the context of studying the rewriting system $\Sigma$, we are working with a distinguished set of generators $X$, so it makes sense to focus on models with $V_0\cong\k X$. Moreover, since 
$H_\bullet(T(V),d)\cong A_{\Sigma_0}$,
the image of $V_1$ under $d$ generates the two-sided ideal of $T(V_0)\cong\k\langle X\rangle$ generated by $W$; it makes sense to focus on models with $V_1\cong\k W$. Using the terminology we recalled earlier, this means that we work with word-homogeneous acyclic extensions of the Shafarevich complex~$\Sha(W,\k\langle X\rangle)$. Finally, we make one slightly less trivial assumption: similar to many situations in rational homotopy theory, we shall consider models of finite type, that is require that $\dim V_n<\infty$ for all $n\ge 0$. This is true for the minimal model of a monomial algebra with finitely many relations, as established by the second author in~\cite{Tam1}, and for some other models of interest.

The following result is a mild generalisation of \cite[Th.~5.1]{Tam1}, which mimics \cite[Th.~4.1]{MR3084563}, and is an adaptation to our case of the argument of Anick~\cite[Th.~1.4]{MR846601}. We say that a linear endomorphism of a word-homogeneous algebra is \emph{$\Sigma$-filtered} if it sends every element of grade $u\in \langle X\rangle$ to a combination of elements of grades strictly less than $u$ with respect to the ordering $<_\Sigma$.

\begin{theorem}\label{th:Perturb}
Suppose that $(T(V),d)$ is a word-homogeneous acyclic extension of finite type of the Shafarevich complex~$\Sha(W,\k\langle X\rangle)$. The rewriting system $\Sigma$ is convergent if and only if there exists an acyclic extension $(T(V),d+d')$ of the Shafarevich complex~$\Sha(R,\k\langle X\rangle)$ with a $\Sigma$-filtered perturbation~$d'$.
\end{theorem}

\noindent
Note that being an acyclic extension of the Shafarevich complex~$\Sha(R,\k\langle X\rangle)$ means in particular that $d'|_{V_0}=0$ and that for each basis element $e_w\subset V_1$, we have
 \[
d'(e_w)=-f(w),
 \]
so $d'$ is indeed $\Sigma$-filtered on elements of degree one. 

\begin{proof}
Suppose first that such a model exists. Extend the partial ordering $<_\Sigma$ to a total well ordering arbitrarily, and consider the filtration of the graded vector space~$T(V)$ associated to that ordering. Since we are dealing with a countable well-order that is not necessarily isomorphic to $\mathbb{N}$, one has to be careful, and either consider a sequence of spectral sequences, or consider generalised spectral sequences of countably filtered modules, such as the transfinite spectral sequences of Hu \cite{MR1718081} or spectral sequences of transfinitely filtered modules of Rahmati \cite{MR3218005}; either strategy works in our case, and allows one to construct a spectral sequence converging to $H_\bullet(T(V),d+d')$ with the zeroth term $E^0$ given by $(T(V),d)$. This forces our spectral sequence to collapse on the following page, since the homology $H_\bullet(T(V),d)$ is concentrated in homological degree zero. Examining the individual word-homogeneous components, we conclude that cosets of the monomial basis of $A_{\Sigma_0}$ form a basis of the vector space $A_{\Sigma}$, and therefore $\Sigma$ is convergent.

Suppose now that $\Sigma$ is convergent. If we ignore the differential, there exist surjective homomorphisms from the algebra $T(V)$ to both the monomial algebra $A_{\Sigma_0}$ and the algebra $A_{\Sigma}$: one may project it onto its part $T(V_0)\cong\k\langle X\rangle$ of homological degree $0$, and the latter admits obvious projection maps to its respective quotients. Let us choose splittings for those projections; this amounts to exhibiting two idempotent endomorphisms $\bar{\pi}$ and $\pi$ of $T(V)$ such that both of them annihilate all elements of positive homological degree, the former annihilates the ideal of $T(V_0)\cong\k\langle X\rangle$ generated by $W$, and the latter annihilates the ideal of $T(V_0)\cong\k\langle X\rangle$ generated by $R$. Since the model $T(V)$ is of finite type, there exists a word-homogeneous map $h\colon T(V) \longrightarrow T(V)$ such that $[d,h] = 1 - \bar{\pi}$.

We are going to define a derivation $d'\colon T(V) \longrightarrow T(V)$ of homological degree~$-1$ satisfying the requested conditions; since $d'$ is a derivation, it is enough to define it on the generators. We shall also define a map \[h'\colon \ker(d+d') \to T(V)\] of homological degree $1$ such that 
 \[\left.(d+d')(h+h')\right|_{\ker(d+d')}=1-\pi .\] This latter condition instantly implies that $(T(V),d+d')$ is a model of $A_\Sigma$.

We shall construct the maps and prove their properties inductively. More specifically, we shall prove by induction on $k\ge 0$ that one can define the values of the perturbation $d'$ on generators of homological degree $k+1$ and the values of $h'$ on elements of $\ker(d+d')$ of homological degree $k$ so that the following conditions hold:

-- \emph{perturbation}: both maps $d'$ (on generators of homological degree $k+1$) and~$h'$ (on elements of $\ker(d+d')$ of homological degree $k$) are $\Sigma$-filtered,

-- \emph{square-zero}: we have $(d+d')^2=0$ on generators of homological degree $k+1$,

-- \emph{homotopy}: we have $(d+d')(h+h')=1-\pi$ on elements of $\ker(d+d')$ of homological degree $k$; of course, according to the definition of the map $\pi$, this condition reads $(d+d')(h+h')=1$ on elements of positive homological degree.

As a basis of induction, we recall that the formula $d'(e_w)=-f(w)$ indicated in the statement of the theorem is $\Sigma$-filtered, and the square-zero condition is satisfied for degree reasons, as there are no elements of negative homological degree. Suppose that $x$ is a degree zero element of~grade $u$. We define
 \[
h'(x)=
\begin{cases}
\qquad \qquad \qquad\qquad 0 ,\ \qquad \quad \qquad \qquad \text{ if } x\notin (W),\\
-(h+h')((1-\pi)(x)-(d+d')h(x)), \text{ if } x\in (W).
\end{cases}
 \]
Since the map $h$ is word-homogeneous and since we already know that $d'$ is $\Sigma$-filtered on elements of homological degree $1$, the map $d'h$ is $\Sigma$-filtered. Also, the map \[(1-\pi)-dh=(1-\pi)-(1-\bar{\pi})=(\bar{\pi}-\pi)\] is $\Sigma$-filtered simply by definition of $<_\Sigma$. By induction with respect to the well-ordering $<_\Sigma$, we may assume the $\Sigma$-filtered condition for the map $h'$ evaluated on sum of those elements, and the perturbation condition for $h'$ evaluated on~$x$ follows. Finally, to show the homotopy property evaluated on elements of homological degree zero, we note that on elements $x\notin (W)$ this becomes $0=0$, and otherwise we have
 \[
(d+d')(h+h')(x)=(d+d')h(x)-(d+d')(h+h')((1-\pi)(x)-(d+d')h(x))
 \]
which, by induction on the well-ordering $<_\Sigma$, is equal to
\begin{multline*}
(d+d')h(x)-(1-\pi)((1-\pi)(x)-(d+d')h(x))=\\ (1-\pi)^2(x)-(1-\pi)(d+d')h(x)=(1-\pi)(x),
\end{multline*}
since $\pi$ vanishes on the image of $d+d'=(R)$ and $1-\pi$ is a projector.

To make the step of induction, we proceed in a similar way. To define $d'(x)$ for a word-homogeneous generator $x$ of homological degree $k+1>1$, we put
 \[
d'(x)=-(h+h')(d+d')d(x).
 \]
The $\Sigma$-filtered property for $d'$ now easily follows by induction. For the square-zero property, we note that 
\begin{multline*}
(d+d')^2(x)=(d+d')(d(x)-(h+h')(d+d')d(x))=\\
(d+d')d(x)-(d+d')(h+h')((d+d')d(x))=\\ (d+d')d(x)-(1-\pi)(d+d')d(x)=\pi((d+d')d(x))=0,
\end{multline*}
since $(d+d')d(x)\in\ker(d+d')$, and since $\pi$ vanishes on the image of $d+d'$. 
Suppose that $x$ is word-homogeneous of grade $u$. Note that whenever $x\in\ker(d+d')$, we have $x-(d+d')h(x)\in\ker(d+d')$ as well. Since $\bar{\pi}$ vanishes on elements of positive homological degree, we have 
$dh(x)=(1-\bar{\pi})(x)=x$, and so, using the perturbation condition for $d'$, we see that $x-(d+d')h(x)$ is a combination of elements of grades strictly less than $u$ with respect to the ordering $<_\Sigma$. Consequently, we may define $h'$ on elements of $\ker(d+d')$ of homological degree $k>0$ by the same inductive argument: we put
 \[
h'(x)=(h+h')(x-(d+d')h(x)).
 \]
Once again, a simple inductive argument shows that the $\Sigma$-filtered property for the map $h'$ and the homotopy condition are satisfied. 

We conclude that $(T(V),d+d')$ a differential graded algebra, with the homology $H_\bullet(T(V),d+d')$ isomorphic to $A_\Sigma$, and a $\Sigma$-filtered map $d'$, as required. It remains to check that it is triangulated. For that, we first decompose the space of generators $V$ according to the homological degree, and then for the fixed homological degree $n$, decompose $V_n$ into word-homogeneous components. Since we work with a model of finite type, there are finitely many word-homogeneous components in each homological degree, and so once we linearly order them extending the ordering $<_\Sigma$, the total order on thus obtained components is either finite or isomorphic to $\mathbb{N}$. This total order gives a triangulation: by construction, the differential of each component is made of products of elements from the previous components. This completes the proof of the theorem. 
\end{proof}

\section{Differential graded Lie algebras and the Diamond Lemma}

In this section, we present a deformation-theoretic version of the Diamond Lemma. This corresponds to the heuristics of deformation theory going back to Deligne, Drinfeld, and Feigin (and later formalised by Lurie \cite{MR2827833,Lur} and Pridham~\cite{MR2628795}) that, infinitesimally, any moduli space of deformations in characteristic zero is controlled by a Maurer--Cartan equation in a suitable differential graded Lie algebra. We remark that while using this heuristics requires to restrict ourselves to a ground field $\k$ of characteristic zero, one can easily see that replacing, in Theorem \ref{th:MCLift}, the Maurer--Cartan equation $[d+F,d+F]=0$ by the square-zero condition $(d+F)^2=0$ leads to a result that is valid without any assumption on the ground field.  

\subsection{Differential graded Lie algebras} We begin with a brief recollection of differential graded Lie algebras and their Maurer--Cartan elements. 

\begin{definition}[Differential graded Lie algebra]
A \emph{differential graded Lie algebra} is a pair $(L,d)$, where $L$ is a graded Lie algebra, that is a graded vector space equipped with a degree zero operation $L\otimes L\to L$, $a\otimes b\mapsto [a,b]$, called the \emph{Lie bracket}, which satisfies the graded skew-symmetry property $[a,b]=-(-1)^{|a||b|}[b,a]$ for homogeneous elements $a,b\in L$ and the graded Jacobi identity
 \[
[a,[b,c]]=[[a,b],c]+(-1)^{|a||b|}[b,[a,c]]
 \]
for homogeneous elements $a,b,c\in L$, and $d\colon L\to L$ is a degree $-1$ map satisfying the condition $d^2=0$ as well as the graded derivation property \[d[a,b]=[d(a),b]+(-1)^{|a|}[a,d(b)].\] 
\end{definition}

From the homotopical algebra point of view, one of the protagonist in the homotopy theory for differential graded Lie algebras is the set of Maurer--Cartan elements of such algebra. 

\begin{definition}
Let $(L,d)$ be a differential graded Lie algebra over a field $\k$ of characteristic different from two. The condition 
 \[
d(x)+\frac12[x,x]=0
 \]
is called the \emph{Maurer--Cartan equation} of $L$, and any solution $x\in L_{-1}$ is called a \emph{Maurer--Cartan element}. 
\end{definition}

Let $L$ be a graded Lie algebra. Note that for an element $x\in L$ of degree $-1$ and any $a\in L$, the Jacobi identity implies
 \[
[x,[x,a]]=[[x,x],a]+(-1)^{|x||x|}[x,[x,a]]=[[x,x],a]-[x,[x,a]],
 \] 
and therefore
 \[
[x,[x,a]]=\frac12[[x,x],a]. 
 \]
Thus, if we suppose that $x$ satisfies the condition $[x,x]=0$ (note that for elements of odd degree the graded antisymmetry property is actually symmetry, so this condition is non-empty), the degree $-1$ endomorphism $[x,-]$ squares to zero; the Jacobi identity implies that it also satisfies the derivation property. Thus, $(L,[x,-])$ is a differential graded Lie algebra.

\subsection{Tangent complexes and deformations}

Our Lie algebras of interest arise as Lie algebras of derivations of models of associative algebras. Suppose that $A$ is an associative algebra, and that $(B,d)$ is a model of $A$. The graded vector space $\Der(B)$ of all derivations of the graded associative algebra $B$ has a natural structure of a graded Lie algebra: the graded bracket 
\[[D_1,D_2]=D_1\circ D_2-(-1)^{|D_1||D_2|}D_2\circ D_1\]
of two derivations is again a derivation. Clearly, $d\in \Der(B)$ and 
 \[
[d,d]=d\circ d-(-1)^{|d||d|}d\circ d=2d^2=0 ;
 \]
as we saw above, the map $[d,-]\colon \Der(B)\to\Der(B)$ makes $\Der(B)$ a differential graded Lie algebra. We note that in the previous section, $d$ denoted the differential of the Lie algebra, and now our context forces us to consider differential graded Lie algebras whose differentials are of the form $[d,-]$ with $d$ being an element of the algebra, the differential of the model; we hope that this does not lead to a confusion.  This algebra is called the \emph{tangent complex} associated to the model $B$. Using the work \cite{MR814187} of Schlessinger and Stasheff as an inspiration, Hinich showed in~\cite{MR1465117} that the homotopy type of this Lie algebra does not depend on the choice of a model. 

In our context, the Maurer--Cartan heuristics of deformation theory acquires very precise meaning. Since the differential of the tangent complex of a model is of the form~$[d,-]$, the Maurer--Cartan equation in that differential graded Lie algebra is equivalent to the equation
$[d+x,d+x]=0$, which in turn is equivalent to the square-zero condition for the derivation $d+x$. Thus, the Maurer--Cartan equation for the differential graded Lie algebra $(\Der(B),[d,-])$ describes perturbations of the differential of $B$. 

\subsection{Perturbing solutions to the Maurer--Cartan equation}

In this section, we shall prove what is essentially the Diamond Lemma in disguise. It claims that having the Maurer--Cartan equation satisfied on generators of homological degree at most two of a model is sufficient to ensure existence of a perturbation. As above, we consider a rewriting system $\Sigma=(X,W,f)$.

\begin{theorem}\label{th:MCLift}
Suppose that $(T(V),d)$ is a word-homogeneous acyclic extension of finite type of the Shafarevich complex~$\Sha(W,\k\langle X\rangle)$. The rewriting system $\Sigma$ is convergent if and only if there exists an degree $-1$ element $F\in\Der(T(V))$ such that 
\begin{itemize}
\item for each $e_w\in V_1$, we have $F(e_w)=-f(w)$,
\item $F|_{V_2}$ is $\Sigma$-filtered,
\item the Maurer--Cartan equation for the element $F$ holds when evaluated on generators of homological degree at most two.
\end{itemize}
\end{theorem}

\begin{proof}
In view of Theorem \ref{th:Perturb}, we have to prove that finding a perturbation that only satisfies the Maurer--Cartan equation on generators of homological degree at most two is enough to ensure that there exists a perturbation that is a Maurer--Cartan element. We shall show by induction on $k$ that we can find a $\Sigma$-filtered derivation $F$ for which the Maurer--Cartan equation holds on generators of homological degree at most $k$, the basis of induction being $k=2$.

Note that $[d+F,d+F]$ is the Lie bracket of derivations, thus itself a derivation, so since it vanishes on generators of homological degree at most~$k$, it vanishes on all elements of $T(V)$ of homological degree at most $k$. Let us consider a word-homogeneous element $c\in V_{k+1}$. We note that the element
 \[ 
\nabla(c)= \left(\partial(F)+\frac12[F,F]\right)(c)=d(F(c))+F(d(c))+F(F(c))\in T(V)
 \] 
is in the kernel of $d+F$. Indeed, we have 
 \[
d(\nabla(c))=d(d(F(c))+F(d(c))+F(F(c)))=(dF)((d+F)(c)),
 \]
and since the element $(d+F)(c)$ is of homological degree at most $k$, the Maurer--Cartan equation for $F$ holds when evaluated on that element. This means that
\begin{multline*}
d(\nabla(c))=-(F d+F^2)((d+F)(c))\\ =-F(d(F(c)))-F(F(d(c)))-F(F(F(c)))=-F(\nabla(c)),
\end{multline*}
as required. The element $\nabla(c)$ is of homological degree $k+1-2=k-1$, and we know that the Maurer--Cartan equation means that the derivation $d+F$ squares to zero on elements of degree at most $k$. Thus, it is not unreasonable to ask whether $\nabla(c)$ is in the image of $d+F$. Let us show that it is indeed the case. We write $\nabla(c)=x+y$, where $x$ is the word-homogeneous component of maximal grade~$u$. Since the differential $d$ is word-homogeneous, and the derivation $F$ is $\Sigma$-filtered on elements of degree at most $k$, it follows that 
 \[
\begin{cases}
d(x)=0,\\
d(y)+F(x+y)=0.
\end{cases}
 \]   
Since $|x|=|\nabla(c)|=k-1\ge 1$ and the model $(T(V),d)$ is acyclic in positive degrees, it follows that we can write $x=da$ for some $a\in T(V)$. The element
 \[ 
y-F(a)= \nabla(c) - (d+F)(a)
 \]
represents the same homology class for $d+F$, but its component of grade $u$ has been replaced by a combination of elements of grades strictly less than $u$ with respect to the ordering $<_\Sigma$. By induction with respect to the well-ordering $<_\Sigma$, we see that $\nabla(c)$ is a boundary of~$d+F$, so we can write $\nabla(c)=(d+F)(b)$. Let us now consider a new derivation $F'$ which coincides with $F$ on elements of homological degree at most $k$, but satisfies
\[F'(c)=F(c)-b.\]
We note that since $\nabla$ is $\Sigma$-filtered, each element $a$ above (and hence the resulting element $b$) is a linear combination of  elements of grades strictly less than the grade of $c$ with respect to the ordering $<_\Sigma$, so $F'$ is still $\Sigma$-filtered. Moreover, since the elements $d(c)$ and $F'(c)$ are of homological degree~$k$, we have
\begin{multline*} 
d(F'(c))+F'(d(c))+F'(F'(c))\\ =d(F(c)-b)+F(d(c))+F(F(c)-b)=\nabla(c)-d(b)-F(b)=0.
\end{multline*}
Making such adjustments for all elements of homological degree $k+1$, we obtain a perturbation that satisfies the Maurer--Cartan equation up to degree $k+1$, which allows us to proceed by induction. Since the correction does not change values of $F$ on generators of low homological degrees, we may consider the limiting value of $F$; it is a \emph{bona fide} Maurer--Cartan element extending the original perturbation of the differential. 
\end{proof}

\subsection{Maurer--Cartan equation and the Diamond Lemma}

In this section, we are finally rewarded for going through all the previous arguments involving models and the Maurer--Cartan equation; the reward is a general result stating that each word-homogeneous acyclic extension of~$\Sha(W,\k\langle X\rangle)$ leads to its own Diamond Lemma criterion for convergence of $\Sigma$. At the core of such result is the following definition of an S-polynomial associated to a generator of homological degree two of such model.  

\begin{definition}[Obstruction]
Let $(T(V),d)$ be a word-homogeneous acyclic extension of the Shafarevich complex~$\Sha(W,\k\langle X\rangle)$. For each element $c\in V_2$, the element $d(c)$ is of degree $1$, and therefore is of the form
 \[
d(c)=\sum_i a_i e_{w_i} b_i\in T(V_0)\otimes V_1\otimes T(V_0).
 \]  
We define the \emph{obstruction} $S_c\in\k\langle X\rangle$ associated to the element $c$ by the formula
 \[
S_c=\sum_i a_i f(w_i) b_i .  
 \]
\end{definition}
We remark that our usage of the term ``obstruction'' is different from that of Anick in \cite{MR846601}: for us, as it will become apparent below, $S_c$ is an actual obstruction to convergence of $\Sigma$.

We are now ready to state and prove our model-specific Diamond Lemma. As above, we consider a rewriting system $\Sigma=(X,W,f)$.

\begin{theorem}\label{thm:Diamond}
Suppose that $(T(V),d)$ is a word-homogeneous acyclic extension of finite type of the Shafarevich complex~$\Sha(W,\k\langle X\rangle)$. The rewriting system $\Sigma$ is convergent if and only if for each word-homogeneous element $c\in V_2$, the corresponding obstruction~$S_c$ is mapped to zero by some reduction with respect to~$\Sigma$.
\end{theorem}

\begin{proof}
According to Theorem \ref{th:MCLift}, $\Sigma$ is convergent if and only if there exists a degree $-1$ element $F\in\Der(T(V))$ such that
\begin{itemize}
\item for each $e_w\subset V_1$, we have $(d+F)(e_w)=w-f(w)$,
\item $F|_{V_2}$ is $\Sigma$-filtered,
\item the Maurer--Cartan condition $[d+F,d+F]=0$ for the element $F$ holds when evaluated on generators of homological degree at most two.
\end{itemize}
Since the Maurer--Cartan equation evaluated on an element of homological degree $k$ is an element of homological degree~$k-2$, we may replace ``at most two'' by ``exactly two'' without changing the statement. 

Suppose $c\in V_2$ is word-homogeneous of grade $u$. We note that the formula \[S_c=\sum_i a_i f(w_i) b_i\] for the obstruction $S_c$ means that $S_c$ is a combination of terms of grade less than~$u$, so existence of its reduction equal to zero is clearly equivalent to existence of a representation of $S_c$ as a two-sided linear combination 
 \[
S_c=\sum_i a_i' (w_i'-f(w_i')) b_i',
 \]
where all the terms in that sum are of grade less than $u$. Indeed, a sequence of basic reductions produces such a combination, and \emph{vice versa}, existence each such combination immediately suggests a sequence of basic reductions. 

Recall that the Maurer--Cartan equation for $F$ is exactly the square-zero condition for $d+F$. Let us evaluate it on an element $c\in V_2$ of grade~$u$. We obtain
 \[
d(d(c))+F(d(c))+d(F(c))+F(F(c))=0.
 \]
We already remarked that the element $d(c)\in T(V)$ is of homological degree one, and therefore can be written as 
 \[
d(c)=\sum_i a_i e_{w_i} b_i \in T(V_0)\otimes V_1\otimes T(V_0).
 \]
The first term in the square-zero condition vanishes, while the second is 
 \[
F(d(c))=F\left(\sum_i a_i e_{w_i} b_i\right)=\sum_i a_i F(e_{w_i}) b_i=-\sum_i a_i f(w_i) b_i=-S_c. 
 \]
The remaining terms in the square-zero condition are 
 \[
d(F(c))+F(F(c))=(d+F)(F(c)).
 \] 
Since $F|_{V_2}$ is $\Sigma$-filtered, the element $F(c)$ is of the form $\sum_i a_i' e_{w_i'} b_i'$, where all the terms in the sum are of grades less than $u$. Finally, we obtain
 \[
d(F(c))+F(F(c))=(d+F)\left(\sum_i a_i' e_{w_i'} b_i'\right)=\sum_i a_i' (w_i'-f(w_i')) b_i'.
 \]
We conclude that the validity of the Maurer--Cartan equation on $c\in V_2$ of grade~$u$ is equivalent to the equation
 \[
S_c=\sum_i a_i' (w_i'-f(w_i')) b_i',
 \]
where all the terms in that sum are of grade less than $u$, which, according to our observation above, completes the proof.
\end{proof}

\section{Examples: specific cases of models and their Diamond Lemmas}

In this section, we discuss two particular examples of models of monomial algebras that allow us to recover the Diamond Lemma and the Triangle Lemma for classical term rewriting. 

\subsection{The inclusion-exclusion model}

In \cite{MR3084563}, the first author and Khoroshkin constructed a combinatorial (usually non-minimal) model for a shuffle operad with monomial relations; specialising to shuffle operads generated by elements of arity~$1$, one obtains a model for an associative algebra with monomial relations. For the benefit of the reader who does not wish to go through the operad construction, we describe that model here. As above, we denote $A_{\Sigma_0}=\k\langle X\mid W\rangle$.

Let $u\in\langle X\rangle$. Suppose that there are exactly $m$ different divisors of the word $u$ that are among the monomial relations $W$. We introduce formal symbols $D_{u,1},\ldots,D_{u,m}$ in one-to-one correspondence with those divisors, and denote by $A_{u}$ the Grassmann algebra $\Lambda(D_{u,1},\ldots,D_{u,m})$. The direct sum 
 \[
A_W:=\bigoplus_{u\in\langle X\rangle} A_u
 \]  
has an associative algebra structure defined as follows. For an element $u'\in\langle X\rangle$ which is a divisor of $u$, there is a natural inclusion $A_{u'}\hookrightarrow A_u$ sending each generator $D_{u',i}$ to the corresponding $D_{u,j}$, where $j$ is the position of the $i$-th divisor of $u'$ (viewed now as a divisor of $u$) in the list of divisors of $u$. The product of two elements $a_1\in A_{u_1}$ and $a_2\in A_{u_2}$ is the product in $A_{u_1u_2}$ of their images with respect to the inclusions $A_{u_1},A_{u_2}\hookrightarrow A_{u_1u_2}$. If we set the homological degree of each $D_{u,i}$ to be equal to one, the algebra $A_W$ is graded. Moreover, it has a differential graded algebra structure, where the differential $d$, when restricted to each Grassmann algebra $A_u$, is the unique derivation of that algebra sending all generators $D_{u,i}$ to~$1$. 

Let us call a Grassmann monomial $D_{u,i_1}\wedge\cdots\wedge D_{u,i_k}\in A_u$ \emph{indecomposable} if there does not exist a factorisation $u=u_1u_2$ into a product of two non-empty monomials $u_1$ and $u_2$ for which each of the divisors of $u$ corresponding to $S_{u,i_p}$ is either a divisor of $u_1$ or a divisor of $u_2$. It is easy to see that as an associative algebra, the algebra $A_W$ is freely generated by its indecomposable elements. Moreover, according to \cite[Th.~2.2]{MR3084563}, the differential graded algebra $(A_W,d)$ is a model of the algebra $A_{\Sigma_0}$. We call this model the \emph{inclusion-exclusion model}. It is immediately seen to be a word-homogeneous acyclic extension of finite type of the Shafarevich complex~$\Sha(W,\k\langle X\rangle)$. Moreover, all homological degree two word-homogeneous generators $c$ of that model of grade $u\in\langle X\rangle$ are of two possible forms:
\begin{itemize}
\item \emph{inclusion}: $c=S_{u,i}\wedge S_{u,j}$ where the $i$-th divisor of $u$ is its proper divisor $u'$ and the $j$-th divisor of $u$ is equal to~$u$,
\item \emph{overlap}: $c=S_{u,i}\wedge S_{u,j}$ where the $i$-th divisor of $u$ is a left divisor $u'$ of $u$, the $j$-th divisor of $u$ is a right divisor $u''$ of $u$, and these two subwords have a non-trivial common divisor.
\end{itemize}
For each generator of the first type, we have $u=au'b$ for some $a,b\in\langle X\rangle$, therefore such generators are in one-to-one correspondence with the inclusion ambiguities of~\cite{Bergman}. We note that for such generator, we have  
 \[d(c)=e_u-ae_{u'}b,\]
and the obstruction associated to this generator is the basic reduction associated to the triple $(a,u',b)$ applied to the relation $u-f(u)$.    
Similarly, for each generator of the second type, we have $u=u'a=bu''$ for some $a,b\in\langle X\rangle$, therefore such generators are in one-to-one correspondence with the overlap ambiguities of \emph{op. cit.}. Moreover, for such generator, we have 
\[d(c)=be_{u''}-e_{u'}a,\] 
and the obstruction associated to this generators is the element usually called the S-polynomial corresponding to the overlap of $u'$ and $u''$ \cite{MR3642294}. This proves the following result.

\begin{proposition}[Bergman's Diamond Lemma {\cite[Th.~1.2]{Bergman}}]
The rewriting system $\Sigma$ is convergent if and only if each basic reduction of each relation from $R$ and each S-polynomial between two  relations from $R$ is mapped to zero by some reduction with respect to~$\Sigma$.
\end{proposition}

As an immediate consequence, we see that the rewriting system \[xyz\mapsto x^3+y^3+z^3\] on $\k\langle x,y,z\rangle$ discussed in Section \ref{sec:Rewriting} is convergent, since we already established its termination, and the word $xyz$ does not have self-overlaps. 

\subsection{The minimal model}

In \cite{Tam1}, the second author constructed, for each monomial algebra $A_{\Sigma_0}=\k\langle X\mid W\rangle$, the minimal model which we shall now recall and use. In this section, the rewriting system $\Sigma$ is assumed minimal: $W$ contains no words of length one and no words from $W$ divide one another. 

We begin with recalling the definition of (right) \emph{Anick chains}, a notion that was to some extent present implicitly in work of Backelin \cite{MR551760}, then used by Green, Happel and Zacharia \cite{MR769766}, and became known to a wider community from the work of Anick \cite{MR846601}. 

\begin{definition}[Anick chain]
We say that an \emph{Anick $0$-chain} is an element $x\in X$, and define the \emph{tail} of such element to be equal to it. 
For $n>0$, we say that an \emph{Anick $n$-chain} is a monomial $c\in\langle X\rangle$ such that
\begin{enumerate}
\item we can write $c = c't$ where $c'$ is an Anick $(n-1)$-chain,
\item if $t'$ is the tail of $c'$, then $t't$ has a right divisor which is a relation from $W$, 
\item no proper left divisor of $c$ satisfies the first two conditions. 
\end{enumerate}
The \emph{tail} of $c$ is the element $t_c := t$. 
\end{definition}

In particular, the Anick $1$-chains are the defining relations $W$, with the tail of each monomial relation given by the monomial obtained by deleting its first letter, and the Anick $2$-chains are precisely the monomials $c\in\langle X\rangle$ which can be written as $c =u'a=bu''$ for two relations $u',u''\in W$, and which do not have other divisors from $W$.

It is shown in \cite{MR846601} that a monomial $c$ admits at most one structure of a chain: if $c$ is an $n$-chain with tail $t_c$, then both $n$ and $t_c$ are uniquely determined. We define the homological degree of an Anick $n$-chain $c$ to be equal to~$n$, and consider the graded associative algebra $B$ freely generated by Anick chains. This algebra has a unique derivation $d$ whose value $d(c)$ on an Anick $n$-chain $c$ is the sum over all ways to represent $c$, viewed as a monomial, as a concatenation $c_1\cdots c_k$ of an Anick $n_1$-chains $c_1$, \ldots, an Anick $n_k$-chain $c_k$ with $n_1+\cdots+n_k=n-1$, of the terms 
 \[
-(-1)^{\binom{n+1}{2}+n_1}c_1\otimes\cdots\otimes c_k.
 \]
The main theorem of \cite{Tam1} states that for a minimal rewriting system $\Sigma$, the quasi-free algebra $(B,d)$ is the minimal model of $A_{\Sigma_0}$. It is a word-homogeneous acyclic extension of finite type of the Shafarevich complex~$\Sha(W,\k\langle X\rangle)$.  

The description of Anick $2$-chains mentioned above indicates that they are among the overlaps from the previous section; in fact, they are \emph{minimal overlaps} that have only two divisors from $W$. For the differential of the model, we have $d(c)=be_{u''}-e_{u'}a$, as before. This proves the following result.

\begin{proposition}[Triangle Lemma {\cite[Sec.~2.4.3]{MR3642294}}]
A minimal rewriting system $\Sigma$ is convergent if and only if each S-polynomial corresponding to a minimal overlap is mapped to zero by some reduction with respect to~$\Sigma$.
\end{proposition}

For example, for the rewriting system $x^3\mapsto xyz-y^3-z^3$ arising in the example discussed in Section \ref{sec:Rewriting} for a total multiplicative ordering, the monomial $x^3$ has two self-overlaps, $x^4=x^3\cdot x=x\cdot x^3$ and $x^5=x^3\cdot x^2=x^2\cdot x^3$. Only the first of them is minimal. This means that there is no need to consider the S-polynomial corresponding to the overlap $x^5$ when checking the Diamond Lemma criterion. That said, in this case already the S-polynomial corresponding to the overlap $x^4$ cannot be reduced to zero, so the rewriting system is not convergent, as we indicated when discussing that example.

\section{Analogues and generalisations}

In words of Kontsevich and Soibelman \cite{MR1805894}, ``the deformation theory of associative algebras is a guide for developing the deformation theory of many algebraic structures; conversely, all the concepts of what should be the ``deformation theory of everything'' must be tested in the case of associative algebras''. We believe that our work constitutes a successful test of putting the Diamond Lemma in the homotopical context, and it is natural to conclude this paper with a discussion of analogues and possible generalisations of our results. 

\subsection{Multiplicative algebraic structures}

We begin with outlining a context in which an immediate generalisation of our approach is available.

\begin{definition}[Multiplicative algebraic structure]
We say that an algebraic structure $\mathsf{P}$ is \emph{multiplicative} if it is described by a coloured operad obtained as a $\k$-linear span of a set-theoretic operad. In classical terms, we require that the free $\k$-linear $\mathsf{P}$-algebra generated by a vector space $V$ is equal to the $\k$-linear span of the free $\mathsf{P}$-algebra in the category of sets generated by a basis of $V$.  
\end{definition}

The key consequence of this definition is that free $\k$-linear $\mathsf{P}$-algebras are naturally equipped with monomial bases, so one can talk about rewriting systems and orderings of monomials, and most of our results apply \emph{mutatis mutandis} if one can talk about models of algebras in the same way as above. Work of Hinich~\cite{MR1465117} mentioned earlier allows one to verify that condition easily enough. 

Some remarks are in order. First, talking about coloured operads means that this formalism includes, for instance, quotients of path algebras, since the latter are obtained as $\k$-linear spans of small categories and are not just algebras in a very classical sense. Second, there are situations where the deviation from the set-theoretic property is ``moderate'', i.e. the value of any structure operation on basis elements is a basis element up to a non-zero constant; it is easy to extend our work to such algebras. Finally, it is also possible to include algebras that are complete in a suitable sense; we already indicated in the introduction that complete commutative associative algebras were one of the central examples at the moment of inception of both topics discussed in this paper, through the theory of standard bases of Hironaka \cite[Sec.~III.1]{MR0199184} and the deformation theory of analytic spaces of Palamodov \cite{MR0508121}.

\subsubsection{Commutative associative algebras}

The free commutative associative algebra coincides with the linear span of the free commutative monoid, and the homotopy theory for commutative associative algebras suggests that our results work over a field of characteristic zero. An analogue of the inclusion-exclusion model exists, and generators of homological degree at most two of the minimal model of a given monomial algebra can be  determined directly (for further information on models of monomial algebras, we refer the reader to \cite{MR2188858}). If the partial order $<_\Sigma$ comes from a total multiplicative order of monomials, the result corresponding to the inclusion-exclusion model is the Buchberger criterion~\cite{MR2202562}, and the result corresponding to the minimal model is known as the Chain Criterion~\cite{MR575678,10.1145/1088261.1088267}. 

\subsubsection{Non-associative algebras}

If one considers free non-associative algebras, better known as magmatic algebras for one binary operation and as absolutely free algebras in the more general case, the free algebra is the linear span of the free set-theoretic magma (consisting of appropriate decorated planar trees), and models  are surprisingly manageable. In particular, the inclusion-exclusion model of a given monomial algebra is readily available, and the minimal model for a monomial algebra whose set of generators and relations are already chosen minimal is simply the non-associative Shafarevich complex, since there are no non-associative overlaps. As a consequence, for a general rewriting system  convergence is equivalent to the fact that all basic reductions between the relations $R$ are mapped to zero by some reduction with respect to~$\Sigma$, and, in particular, a rewriting system where pairwise reductions are impossible is convergent. This result goes back to \cite{MR0037831}, see also \cite[Appendix~3]{MR1292459} and a more recent paper \cite{MR2202553}. 

\subsubsection{Shuffle algebras}

Shuffle algebras \cite{MR3095223,MR2918719}, known also as permutads \cite{MR2995045}, are associative algebras for the monoidal category of non-symmetric sequences with respect to the monoidal structure given by Cauchy tensor product. The corresponding definition is also available for sets; elements of the free monoid are decorated permutations and their suitable generalisations. All results of this paper, including a generalisation of the inclusion-exclusion model and the second author's minimal model, can be adapted; the corresponding results are the Diamond Lemma for shuffle algebras \cite[Th.~4.5.1.4]{MR3642294} and the Triangle Lemma for shuffle algebras \cite[Prop.~4.5.3.2]{MR3642294}.

\subsubsection{Non-symmetric operads and shuffle operads}

When one considers operads, whether non-symmetric or shuffle ones, each free operad is the linear span of the corresponding free set operad consisting of appropriate decorated planar trees. It is important to note that the algebraic structure on the space of decorated trees in this case is much richer than the one considered in the case of magmatic algebras above, so the corresponding Diamond Lemmas (and especially their applications) are sufficiently non-trivial. Models are once again available under relatively mild assumptions \cite{MR1380606}. In particular, the inclusion-exclusion model can be defined for every shuffle operad with monomial relations \cite[Sec.~2]{MR3084563}. This way one obtains the Diamond Lemma for non-symmetric operads and for shuffle operads \cite{MR3642294}. 

Finding an explicit description of the minimal model of the given operad with monomial relations is an open problem. Partial results in the spirit of the Triangle Lemma are available for operads~\cite{MR3642294}, but even in homological degree two, the computation of the spaces of generators of minimal models of monomial operads still has to be completed in full generality. One calculation hinting at the noticeable complexity of this question was made by Sk\"oldberg several years ago (unpublished), and we shall now describe the result of that calculation; another discussion of that phenomenon can be found in the recent preprint of Iyudu and Vlassopoulos~\cite{iyudu2020homologies}. 

The minimal model of a monomial operad is homogeneous: one can separate generators according to their underlying tree monomials. In the case of associative algebras, minimal models of monomial algebras satisfy the homological purity condition: for each monomial, the corresponding component of the space of generators is concentrated in one homological degree. Somewhat surprisingly, it turns out that in the case of operads, this homological purity condition does not hold for some underlying tree monomials. The simplest possible example is obtained as follows. Consider the free non-symmetric operad generated by one binary operation, and consider the set of monomial relations $W$ consisting of all tree monomials of arity four in that operad. Then for the tree monomial
 \[
\vcenter
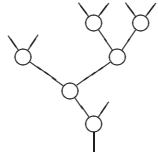
{\hbox{\xymatrix@R=.4pc@C=.2pc{ %}@R=.7mm@C=1mm{
& & & \ar@{-}[dr]&&\ar@{-}[dl] \ar@{-}[dr]& &\ar@{-}[dl] \\
\ar@{-}[dr] & &\ar@{-}[dl]& &*+[o][F-]{}\ar@{-}[dr]& {} &*+[o][F-]{}\ar@{-}[dl]& {} \\
&*+[o][F-]{}\ar@{-}[drr]& & & &*+[o][F-]{}\ar@{-}[dll]  & \\
&  & &*+[o][F-]{}\ar@{-}[dr] & &\ar@{-}[dl] \\
&  & & &*+[o][F-]{}\ar@{-}[d] & \\
& & &*{} & \\
}}}
 \]
of arity seven, as well as the monomials obtained from it by reflections at internal vertices, the corresponding component of the space generators of the minimal model is two-dimensional, with one generator in degree three and one generator in degree four. This example is a particular case of the remarkable construction from the very recent preprint of Qi, Xu, Zhang and Zhao \cite[Sec.~5]{qi2020growth}: it is the ``operadization'' of the algebra $\k\langle x,y\rangle/\mathfrak{m}^3$. We feel that explicitly computing minimal models for monomial operads obtained by that construction is likely to exhibit some of the most interesting phenomena behind minimal models of general monomial operads.

\subsection{Non-multiplicative algebraic structures}

For an algebraic structure that is not multiplicative, our approach does not adapt immediately. Upon examining the existing literature on Gr\"obner--Shirshov bases, see, for example, \cite{MR2386984,MR3277879,MR3328636} and references therein, we feel that the most obvious way to deal with a non-multiplicative algebraic structure is to embed it into a bigger multiplicative one. For example, Lie algebras can be embedded in their universal envelopes, and so one can work with a rewriting system within a multiplicative algebraic structure. Equivalence of the associative algebra methods applied to the universal envelopes and the Lie-algebraic methods of Shirshov was established by Bokut and Malcolmson \cite{MR1733167}. Alternatively, a Lie algebra can be regarded as an anti-commutative magmatic algebra; the corresponding study of Gr\"obner--Shirshov bases was undertaken in \cite{MR2491724}. Recent papers on Gr\"obner--Shirshov bases for pre-Lie algebras \cite{MR2744933}, Novikov algebras \cite{MR3847122}, and Sabinin algebras \cite{MR4150721} follow the same logic. This raises a question of computing not only models of monomial algebras for multiplicative algebraic structures, but also models of meaningful classes of algebras obtained when extending non-multiplicative structures to multiplicative ones. We hope to address this in more detail elsewhere.

\bibliographystyle{plain} 
\bibliography{mc-equations-and-rewriting.bib}

\def\cprime{$'$}
\begin{thebibliography}{10}

\bibitem{MR644015}
David~J. Anick.
\newblock A counterexample to a conjecture of {S}erre.
\newblock {\em Ann. of Math. (2)}, 115(1):1--33, 1982.

\bibitem{MR677714}
David~J. Anick.
\newblock Noncommutative graded algebras and their {H}ilbert series.
\newblock {\em J. Algebra}, 78(1):120--140, 1982.

\bibitem{MR846601}
David~J. Anick.
\newblock On the homology of associative algebras.
\newblock {\em Trans. Amer. Math. Soc.}, 296(2):641--659, 1986.

\bibitem{MR551760}
J\"{o}rgen Backelin.
\newblock La s\'{e}rie de {P}oincar\'{e}-{B}etti d'une alg\`ebre gradu\'{e}e de
  type fini \`a une relation est rationnelle.
\newblock {\em C. R. Acad. Sci. Paris S\'{e}r. A-B}, 287(13):A843--A846, 1978.

\bibitem{MR1874282}
Michael~J. Bardzell.
\newblock Noncommutative {G}r\"{o}bner bases and {H}ochschild cohomology.
\newblock In {\em Symbolic computation: solving equations in algebra, geometry,
  and engineering ({S}outh {H}adley, {MA}, 2000)}, volume 286 of {\em Contemp.
  Math.}, pages 227--240. Amer. Math. Soc., Providence, RI, 2001.

\bibitem{MR2694031}
Michael~Joseph Bardzell.
\newblock {\em Resolutions and cohomology of finite dimensional algebras.
  Thesis (Ph.D.)--Virginia Polytechnic Institute and State University}.
\newblock ProQuest LLC, Ann Arbor, MI, 1996.

\bibitem{BarWang}
Severin Barmeier and Zhengfang Wang.
\newblock Deformations of path algebras of quivers with relations.
\newblock {\em arXiv e-print arXiv:2002.10001}, February 2020.

\bibitem{MR431172}
H.~J. Baues and J.-M. Lemaire.
\newblock Minimal models in homotopy theory.
\newblock {\em Math. Ann.}, 225(3):219--242, 1977.

\bibitem{MR2188858}
Alexander Berglund.
\newblock Poincar\'{e} series of monomial rings.
\newblock {\em J. Algebra}, 295(1):211--230, 2006.

\bibitem{MR3276839}
Alexander Berglund.
\newblock Homological perturbation theory for algebras over operads.
\newblock {\em Algebr. Geom. Topol.}, 14(5):2511--2548, 2014.

\bibitem{Bergman}
George~M. Bergman.
\newblock The diamond lemma for ring theory.
\newblock {\em Adv. in Math.}, 29(2):178--218, 1978.

\bibitem{MR2386984}
L.~A. Bokut and Yuqun Chen.
\newblock Gr\"{o}bner-{S}hirshov bases for {L}ie algebras: after {A}. {I}.
  {S}hirshov.
\newblock {\em Southeast Asian Bull. Math.}, 31(6):1057--1076, 2007.

\bibitem{MR3277879}
L.~A. Bokut and Yuqun Chen.
\newblock Gr\"{o}bner-{S}hirshov bases and their calculation.
\newblock {\em Bull. Math. Sci.}, 4(3):325--395, 2014.

\bibitem{MR3328636}
L.~A. Bokut and Yuqun Chen.
\newblock Gr\"{o}bner-{S}hirshov bases for universal algebras.
\newblock {\em J. South China Normal Univ. Natur. Sci. Ed.}, 46(6):1--9, 2014.

\bibitem{MR2491724}
L.~A. Bokut, YuQun Chen, and Yu~Li.
\newblock Anti-commutative {G}r\"{o}bner-{S}hirshov basis of a free {L}ie
  algebra.
\newblock {\em Sci. China Ser. A}, 52(2):244--253, 2009.

\bibitem{MR1733167}
L.~A. Bokut and P.~Malcolmson.
\newblock Gr\"{o}bner-{S}hirshov bases for relations of a {L}ie algebra and its
  enveloping algebra.
\newblock In {\em Algebras and combinatorics ({H}ong {K}ong, 1997)}, pages
  47--54. Springer, Singapore, 1999.

\bibitem{MR0506423}
L.~A. Bokut\cprime.
\newblock Imbeddings into simple associative algebras.
\newblock {\em Algebra i Logika}, 15(2):117--142, 245, 1976.

\bibitem{MR2744933}
L.~A. Bokut\cprime, Yu\u{\i}tsyun\cprime Ch\`en\cprime, and Yu\u{\i} Li.
\newblock Gr\"{o}bner-{S}hirshov bases {V}inberg-{K}oszul-{G}erstenhaber for
  right-symmetric algebras.
\newblock {\em Fundam. Prikl. Mat.}, 14(8):55--67, 2008.

\bibitem{MR1292459}
L.~A. Bokut\cprime and G.~P. Kukin.
\newblock {\em Algorithmic and combinatorial algebra}, volume 255 of {\em
  Mathematics and its Applications}.
\newblock Kluwer Academic Publishers Group, Dordrecht, 1994.

\bibitem{MR3642294}
Murray~R. Bremner and Vladimir Dotsenko.
\newblock {\em Algebraic operads: an algorithmic companion}.
\newblock CRC Press, Boca Raton, FL, 2016.

\bibitem{MR575678}
B.~Buchberger.
\newblock A criterion for detecting unnecessary reductions in the construction
  of {G}r\"{o}bner-bases.
\newblock In {\em Symbolic and algebraic computation ({EUROSAM} '79,
  {I}nternat. {S}ympos., {M}arseille, 1979)}, volume~72 of {\em Lecture Notes
  in Comput. Sci.}, pages 3--21. Springer, Berlin-New York, 1979.

\bibitem{MR2202562}
Bruno Buchberger.
\newblock An algorithm for finding the basis elements of the residue class ring
  of a zero dimensional polynomial ideal.
\newblock {\em J. Symbolic Comput.}, 41(3-4):475--511, 2006.
\newblock Translated from the 1965 German original by Michael P. Abramson.

\bibitem{MR3334140}
Sergio Chouhy and Andrea Solotar.
\newblock Projective resolutions of associative algebras and ambiguities.
\newblock {\em J. Algebra}, 432:22--61, 2015.

\bibitem{MR1897811}
Maurizio~G. Citterio.
\newblock Classifying spaces of categories and term rewriting.
\newblock {\em Theory Appl. Categ.}, 9:92--105, 2001/02.

\bibitem{MR3084563}
Vladimir Dotsenko and Anton Khoroshkin.
\newblock Quillen homology for operads via {G}r\"{o}bner bases.
\newblock {\em Doc. Math.}, 18:707--747, 2013.

\bibitem{MR3095223}
Vladimir Dotsenko and Anton Khoroshkin.
\newblock Shuffle algebras, homology, and consecutive pattern avoidance.
\newblock {\em Algebra Number Theory}, 7(3):673--700, 2013.

\bibitem{10.1145/2631948.2631968}
Christian Eder.
\newblock Predicting zero reductions in gr\"{o}bner basis computations.
\newblock In {\em Proceedings of the 2014 Symposium on Symbolic-Numeric
  Computation}, page 109–110. Association for Computing Machinery, New York,
  NY, 2014.

\bibitem{MR2202553}
Lothar Gerritzen.
\newblock Tree polynomials and non-associative {G}r\"{o}bner bases.
\newblock {\em J. Symbolic Comput.}, 41(3-4):297--316, 2006.

\bibitem{Golod}
E.~S. Golod.
\newblock Standard bases and homology.
\newblock In {\em Algebra---some current trends ({V}arna, 1986)}, volume 1352
  of {\em Lecture Notes in Math.}, pages 88--95. Springer, Berlin, 1988.

\bibitem{MR1799542}
E.~S. Golod.
\newblock On the homology algebra of the {S}hafarevich complex of a free
  algebra.
\newblock {\em Fundam. Prikl. Mat.}, 5(1):97--100, 1999.

\bibitem{MR0161852}
E.~S. Golod and I.~R. \v{S}afarevi\v{c}.
\newblock On the class field tower.
\newblock {\em Izv. Akad. Nauk SSSR Ser. Mat.}, 28:261--272, 1964.

\bibitem{MR769766}
E.~L. Green, D.~Happel, and D.~Zacharia.
\newblock Projective resolutions over {A}rtin algebras with zero relations.
\newblock {\em Illinois J. Math.}, 29(1):180--190, 1985.

\bibitem{guetta2020homology}
L\'eonard Guetta.
\newblock Homology of categories via polygraphic resolutions.
\newblock {\em arXiv e-print arXiv:2003.10734}, Mar 2020.

\bibitem{MR1007895}
V.~K. A.~M. Gugenheim and L.~A. Lambe.
\newblock Perturbation theory in differential homological algebra. {I}.
\newblock {\em Illinois J. Math.}, 33(4):566--582, 1989.

\bibitem{MR1103672}
V.~K. A.~M. Gugenheim, L.~A. Lambe, and J.~D. Stasheff.
\newblock Perturbation theory in differential homological algebra. {II}.
\newblock {\em Illinois J. Math.}, 35(3):357--373, 1991.

\bibitem{MR885535}
V.~K. A.~M. Gugenheim and J.~D. Stasheff.
\newblock On perturbations and {$A_\infty$}-structures.
\newblock {\em Bull. Soc. Math. Belg. S\'{e}r. A}, 38:237--246 (1987), 1986.

\bibitem{MR4002273}
Yves Guiraud, Eric Hoffbeck, and Philippe Malbos.
\newblock Convergent presentations and polygraphic resolutions of associative
  algebras.
\newblock {\em Math. Z.}, 293(1-2):113--179, 2019.

\bibitem{MR1465117}
Vladimir Hinich.
\newblock Homological algebra of homotopy algebras.
\newblock {\em Comm. Algebra}, 25(10):3291--3323, 1997.

\bibitem{MR0175898}
Heisuke Hironaka.
\newblock On resolution of singularities (characteristic zero).
\newblock In {\em Proc. {I}nternat. {C}ongr. {M}athematicians ({S}tockholm,
  1962)}, pages 507--521. Inst. Mittag-Leffler, Djursholm, 1963.

\bibitem{MR0199184}
Heisuke Hironaka.
\newblock Resolution of singularities of an algebraic variety over a field of
  characteristic zero. {I}, {II}.
\newblock {\em Ann. of Math. (2) {\bf 79} (1964), 109--203; ibid. (2)},
  79:205--326, 1964.

\bibitem{MR1650134}
Mark Hovey.
\newblock {\em Model categories}, volume~63 of {\em Mathematical Surveys and
  Monographs}.
\newblock American Mathematical Society, Providence, RI, 1999.

\bibitem{MR1718081}
Po~Hu.
\newblock Transfinite spectral sequences.
\newblock In {\em Homotopy invariant algebraic structures ({B}altimore, {MD},
  1998)}, volume 239 of {\em Contemp. Math.}, pages 197--216. Amer. Math. Soc.,
  Providence, RI, 1999.

\bibitem{MR1109665}
Johannes Huebschmann and Tornike Kadeishvili.
\newblock Small models for chain algebras.
\newblock {\em Math. Z.}, 207(2):245--280, 1991.

\bibitem{iyudu2020homologies}
Natalia Iyudu and Ioannis Vlassopoulos.
\newblock Homologies of monomial operads and algebras.
\newblock {\em arXiv e-print arXiv:2008.00985}, August 2020.

\bibitem{MR1478701}
J.~F. Jardine.
\newblock A closed model structure for differential graded algebras.
\newblock In {\em Cyclic cohomology and noncommutative geometry ({W}aterloo,
  {ON}, 1995)}, volume~17 of {\em Fields Inst. Commun.}, pages 55--58. Amer.
  Math. Soc., Providence, RI, 1997.

\bibitem{MR268172}
T.~J\'{o}zefiak.
\newblock Tate resolutions for commutative graded algebras.
\newblock {\em Bull. Acad. Polon. Sci. S\'{e}r. Sci. Math. Astronom. Phys.},
  17:617--621, 1969.

\bibitem{Kel}
Bernhard Keller.
\newblock Notes on minimal models.
\newblock {\em Preprint available on the author's webpage
  {\rm\url{https://webusers.imj-prg.fr/~bernhard.keller/publ/index.html}}},
  2003.

\bibitem{10.1145/1088261.1088267}
C.~Kollreider and B.~Buchberger.
\newblock An improved algorithmic construction of gr\"{o}bner-bases for
  polynomial ideals.
\newblock {\em SIGSAM Bull.}, 12(2):27–36, May 1978.

\bibitem{MR1805894}
Maxim Kontsevich and Yan Soibelman.
\newblock Deformations of algebras over operads and the {D}eligne conjecture.
\newblock In {\em Conf\'{e}rence {M}osh\'{e} {F}lato 1999, {V}ol. {I}
  ({D}ijon)}, volume~21 of {\em Math. Phys. Stud.}, pages 255--307. Kluwer
  Acad. Publ., Dordrecht, 2000.

\bibitem{MR1324032}
Yves Lafont.
\newblock A new finiteness condition for monoids presented by complete
  rewriting systems (after {C}raig {C}. {S}quier).
\newblock {\em J. Pure Appl. Algebra}, 98(3):229--244, 1995.

\bibitem{MR2498787}
Yves Lafont and Fran\c{c}ois M\'{e}tayer.
\newblock Polygraphic resolutions and homology of monoids.
\newblock {\em J. Pure Appl. Algebra}, 213(6):947--968, 2009.

\bibitem{MR1146597}
Yves Lafont and Alain Prout\'{e}.
\newblock Church-{R}osser property and homology of monoids.
\newblock {\em Math. Structures Comput. Sci.}, 1(3):297--326, 1991.

\bibitem{MR1187288}
Larry~A. Lambe.
\newblock Homological perturbation theory, {H}ochschild homology, and formal
  groups.
\newblock In {\em Deformation theory and quantum groups with applications to
  mathematical physics ({A}mherst, {MA}, 1990)}, volume 134 of {\em Contemp.
  Math.}, pages 183--218. Amer. Math. Soc., Providence, RI, 1992.

\bibitem{Lat88}
V.~N. Latyshev.
\newblock {\em Combinatorial ring theory. Standard bases.}
\newblock Izdatel\cprime stvo MGU, 1988.

\bibitem{MR1754671}
V.~N. Latyshev.
\newblock General version of standard bases in linear structures.
\newblock In {\em Algebra ({M}oscow, 1998)}, pages 215--226. de Gruyter,
  Berlin, 2000.

\bibitem{MR2128915}
V.~N. Latyshev.
\newblock A general version of standard basis and its application to
  {$T$}-ideals.
\newblock {\em Acta Appl. Math.}, 85(1-3):219--223, 2005.

\bibitem{MR2744977}
V.~N. Latyshev.
\newblock A general version of a standard basis in associative algebras and
  their derivative constructions.
\newblock {\em Fundam. Prikl. Mat.}, 15(3):183--203, 2009.

\bibitem{MR0142595}
V.~N. Laty\v{s}ev.
\newblock Algebras with identical relations.
\newblock {\em Dokl. Akad. Nauk SSSR}, 146:1003--1006, 1962.

\bibitem{MR0156874}
V.~N. Laty\v{s}ev.
\newblock On the choice of basis in a {$T$}-ideal.
\newblock {\em Sibirsk. Mat. \v{Z}.}, 4:1122--1127, 1963.

\bibitem{MR500930}
J.-M. Lemaire.
\newblock ``{A}utopsie d'un meurtre'' dans l'homologie d'une alg\`ebre de
  cha\^{\i}nes.
\newblock {\em Ann. Sci. \'{E}cole Norm. Sup. (4)}, 11(1):93--100, 1978.

\bibitem{MR4150721}
Yu~Li, Qiuhui Mo, and L.~A. Bokut.
\newblock Generalized anti-commutative {G}r\"{o}bner-{S}hirshov basis theory
  and free {S}abinin algebras.
\newblock {\em Comm. Algebra}, 48(12):5086--5109, 2020.

\bibitem{MR2995045}
Jean-Louis Loday and Mar\'{\i}a Ronco.
\newblock Permutads.
\newblock {\em J. Combin. Theory Ser. A}, 120(2):340--365, 2013.

\bibitem{MR2827833}
Jacob Lurie.
\newblock Moduli problems for ring spectra.
\newblock In {\em Proceedings of the {I}nternational {C}ongress of
  {M}athematicians. {V}olume {II}}, pages 1099--1125. Hindustan Book Agency,
  New Delhi, 2010.

\bibitem{Lur}
Jacob Lurie.
\newblock Derived algebraic geometry {X}: {F}ormal moduli problems.
\newblock {\em Preprint available on the author's webpage
  {\rm\url{https://www.math.ias.edu/~lurie/}}}, 2011.

\bibitem{MR2483955}
Marco Manetti.
\newblock Differential graded {L}ie algebras and formal deformation theory.
\newblock In {\em Algebraic geometry---{S}eattle 2005. {P}art 2}, volume~80 of
  {\em Proc. Sympos. Pure Math.}, pages 785--810. Amer. Math. Soc., Providence,
  RI, 2009.

\bibitem{MR1380606}
Martin Markl.
\newblock Models for operads.
\newblock {\em Comm. Algebra}, 24(4):1471--1500, 1996.

\bibitem{MR1988395}
Fran\c{c}ois M\'{e}tayer.
\newblock Resolutions by polygraphs.
\newblock {\em Theory Appl. Categ.}, 11:No. 7, 148--184, 2003.

\bibitem{MR7372}
M.~H.~A. Newman.
\newblock On theories with a combinatorial definition of ``equivalence.''.
\newblock {\em Ann. of Math. (2)}, 43:223--243, 1942.

\bibitem{MR0508121}
V.~P. Palamodov.
\newblock Deformations of complex spaces.
\newblock {\em Uspehi Mat. Nauk}, 31(3):129--194, 1976.

\bibitem{MR265437}
Stewart~B. Priddy.
\newblock Koszul resolutions.
\newblock {\em Trans. Amer. Math. Soc.}, 152:39--60, 1970.

\bibitem{MR2628795}
J.~P. Pridham.
\newblock Unifying derived deformation theories.
\newblock {\em Adv. Math.}, 224(3):772--826, 2010.

\bibitem{qi2020growth}
Zihao Qi, Yongjun Xu, James~J. Zhang, and Xiangui Zhao.
\newblock Growth of nonsymmetric operads.
\newblock {\em arXiv e-print arXiv:2010.12172}, Oct 2020.

\bibitem{MR0223432}
Daniel~G. Quillen.
\newblock {\em Homotopical algebra}.
\newblock Lecture Notes in Mathematics, No. 43. Springer-Verlag, Berlin-New
  York, 1967.

\bibitem{MR3218005}
Saeed Rahmati.
\newblock {\em Theory of {S}pectral {S}equences of {E}xact {C}ouples:
  {A}pplications {T}o {C}ountably {A}nd {T}ransfinitely {F}iltered {M}odules}.
\newblock ProQuest LLC, Ann Arbor, MI, 2013.
\newblock Thesis (Ph.D.)--University of Alberta (Canada).

\bibitem{RedBer}
Mar\'ia~Julia Redondo and Fiorella~Rossi Bertone.
\newblock ${L}_\infty$-structure on {B}ardzell's complex for monomial algebras.
\newblock {\em arXiv e-print arXiv:2008.08122}, August 2020.

\bibitem{MR2918719}
Mar\'{\i}a Ronco.
\newblock Shuffle bialgebras.
\newblock {\em Ann. Inst. Fourier (Grenoble)}, 61(3):799--850, 2011.

\bibitem{MR814187}
Michael Schlessinger and James Stasheff.
\newblock The {L}ie algebra structure of tangent cohomology and deformation
  theory.
\newblock {\em J. Pure Appl. Algebra}, 38(2-3):313--322, 1985.

\bibitem{SchSta}
Mike Schlessinger and Jim Stasheff.
\newblock Deformation theory and rational homotopy type.
\newblock {\em arXiv e-print arXiv:1211.1647}, November 2012.

\bibitem{MR35274}
Wilhelm Specht.
\newblock Gesetze in {R}ingen. {I}.
\newblock {\em Math. Z.}, 52:557--589, 1950.

\bibitem{MR920522}
Craig~C. Squier.
\newblock Word problems and a homological finiteness condition for monoids.
\newblock {\em J. Pure Appl. Algebra}, 49(1-2):201--217, 1987.

\bibitem{MR646078}
Dennis Sullivan.
\newblock Infinitesimal computations in topology.
\newblock {\em Inst. Hautes \'{E}tudes Sci. Publ. Math.}, (47):269--331 (1978),
  1977.

\bibitem{Tam1}
Pedro Tamaroff.
\newblock Minimal models for monomial algebras.
\newblock {\em Homology Homotopy Appl., {\rm to appear}}, 2021.

\bibitem{MR86072}
John Tate.
\newblock Homology of {N}oetherian rings and local rings.
\newblock {\em Illinois J. Math.}, 1:14--27, 1957.

\bibitem{MR3847122}
Rabigul Tuniyaz, L.~A. Bokut, Marhaba Xiryazidin, and Abdukadir Obul.
\newblock Gr\"{o}bner-{S}hirshov bases for free {G}elfand-{D}orfman-{N}ovokov
  algebras and for right ideals of free right {L}eibniz algebras.
\newblock {\em Comm. Algebra}, 46(10):4392--4402, 2018.

\bibitem{MR1360005}
V.~A. Ufnarovskij.
\newblock Combinatorial and asymptotic methods in algebra.
\newblock In {\em Algebra, {VI}}, volume~57 of {\em Encyclopaedia Math. Sci.},
  pages 1--196. Springer, Berlin, 1995.

\bibitem{MR0183753}
A.~I. \v{S}ir\v{s}ov.
\newblock Some algorithm problems for {L}ie algebras.
\newblock {\em Sibirsk. Mat. \v{Z}.}, 3:292--296, 1962.

\bibitem{MR0037831}
A.~I. \v{Z}ukov.
\newblock Reduced systems of defining relations in non-associative algebras.
\newblock {\em Mat. Sbornik N.S.}, 27(69):267--280, 1950.

\end{thebibliography}

\end{document}